\documentclass[a4paper,12pt]{article}
\usepackage{amsrefs,amstext,amsmath,amsthm,amsfonts,amssymb,enumerate,amscd,makeidx,mathrsfs}
\usepackage[refpage,intoc]{nomencl}
\usepackage{hyperref}
\usepackage[dvipdf]{graphicx}
\usepackage[margin=3 true cm]{geometry}
\usepackage{todonotes}
\usepackage{xspace}

\newtheorem{theorem}{Theorem}[section]
\newtheorem{proposition}[theorem]{Proposition}
\newtheorem{lemma}[theorem]{Lemma}
\newtheorem{corollary}[theorem]{Corollary}

\newtheorem{definition}[theorem]{Definition}
\newtheorem{remark}[theorem]{Remark}

\makenomenclature

\title{(First Draft) Amenability, Nuclearity and Tensor Products of $C^{\ast}$-Algebraic Fell Bundles under the Unified Viewpoint of  the Fell-Doran Induced Representation Theory}
\author{Weijiao He}

\begin{document}

\maketitle

\begin{abstract}
In this paper we study amenability, nuclearity and tensor products of $C^{\ast}$-Fell bundles by the method of induced representation theory.
\end{abstract}

\section*{Introduction}

$C^{\ast}$-algebraic Fell bundles over locally compact groups, which we denoted by $\mathscr{B}$,  were defined and studied by Fell and Doran in \cite{MR936629}. A lot of definitions and theorems on locally compact groups can be generalized on $\mathscr{B}$. In \cite{MR1891686}, Exel and Ng studied the amenability and the approximation property (AP) of $\mathscr{B}$, which were the generalization of the amenability and approximation property of locally compact groups. In the recent years, the problems about $\mathscr{B}$ which are in the connection with amenability, tensor products and nuclearity of $C^{\ast}$-algebras are studied under varied conditions: see Abadie, Buss and Ferraro \cite{vjnj} for the case that $\mathscr{B}$ is over locally compact groups; Buss, Echterhoff and Willett \cite{vnkj} for the case that $\mathscr{B}$ is semidirect product bundles over locally compact groups; Abadie-Vicens \cite{hhncjdl}, Abadie, Buss and Ferraro \cite{nkjdklkwa}, Ara, Exel, and Katsura \cite{aaab}, Exel \cite{MR3699795}, He \cite{hsfh}, McKee, Skalski, Todorov and Turowska\cite{MR3860401} for the case that $\mathscr{B}$ is over discrete groups;  Lalonde \cite{cnkjllkew} and  Sims and Williams \cite{cnnjlilid} \cite{fhkldsll} for the case that $\mathscr{B}$ is separable and is over secound countable $groupoids$; Takeishi \cite{aaaaa} for the case that $\mathscr{B}$ is over $\acute{E}$$tale$ $groupoids$. 

The main objective of this paper is, letting $\mathscr{B}$ to be a $C^{\ast}$-algebraic Fell bundle over locally compact groups, to study the relations between the amenability of $\mathscr{B}$, tensor products of $\mathscr{B}$ and $C^{\ast}$-algebras, and the nuclearity of the full or regular $C^{\ast}$-algebras of $\mathscr{B}$, by the method of induced representation theory developed by Fell and Doran in \cite{MR936629}. Corollary \ref{kjdsnf} and Theorem \ref{vnnjsdliclsasod} give partial solution to Claire Anantharaman-Delaroche \cite[Problem 9.2(d)]{ejjlklds}, which is partially the motivation of this paper.

In section 1, we give a brief review of the theory of induced representation of $C^{\ast}$-algebraic Fell bundles and Fell-Doran's Imprimitivity Theorem.

In section 2, we prove that the Moraita equivalence constructed by Fell-Doran's Imprimitivity Theorem preserves the regular representations. 

In section 3, we develop the theory of the tensor products of $C^{\ast}$-algebraic Fell bundles and $C^{\ast}$-algebras. In the case that $G$ is discrete, this theory is developed in some papers we referred in the first paragraph. Our treatment based on the induced representation theory is new, strikingly easier, and generalize the results which hold for discrete bundles.

In section 4, we give the definition of ultra-approximation property (UAP) of $C^{\ast}$-algebraic Fell bundles, which is weaker than the approximation property and is a sufficient condition by which the bundle is amenable. Based on the results of Proposition \ref{vnnisfdliasiooa} and Proposition \ref{cnhjfdskhjfsd}, in Remark \ref{vkjliidasdsjigfo} we show that there are plenty ``non-trivial'' examples of $C^{\ast}$-algebraic bundles having UAP, particularly are amenable. Finally, we combine Proposition \ref{vnnisfdliasiooa} and Proposition \ref{cnhjfdskhjfsd} with the results of section 2 and section 3 to prove our main result, i.e. Theorem \ref{vnnjsdliclsasod}.

\section{Background}
Throughout this paper, $G$ is a locally compact group and we choose once for all a Haar measure on $G$, $R^G$ is the left-regular representation of $G$, and $H \subset G$ a closed subgroup of $G$. In addition, we will choose once for all a continuous everywhere positive $H$-rho function $\rho$ on $G$ and denote by $\rho^{\sharp}$ the regular Borel measure on $G/H$ constructed from $\rho$. Assume $\mathscr{B}$ is a $C^{\ast}$-algebraic Fell bundle over $G$ with fibers $\{B_x\}_{x \in G}$, and let $\mathscr{B}_H$ be the restricted bundle of $\mathscr{B}$ to $H$. Let $\mathscr{L}(\mathscr{B})$, $\mathscr{L}_1(\mathscr{B})$ and $\mathscr{L}_2(\mathscr{B})$ to denote the space of continuous cross-sections vanishing at infinite point, absolute integral cross-sections and squared-integrable cross-sections respectively. For details on the definition of $C^{\ast}$-algebraic Fell bundles over locally compact group $G$ and the cross-sectional algebras we refer the reader to \cite[\S Chapter VIII]{MR936629}.

For $\ast$-algebra $A$ ( resp. $\ast$-algebraic bundle $\mathscr{B}$), let $T$ and $S$ be $\ast$-representations of $A$ ( resp.$\mathscr{B}$), we use symbol $T \leq S$ to indicate that $S$ weakly contains $T$.

 If $T$ is $\ast$-representation of $C^{\ast}$-algebraic bundle $\mathscr{B}$, we use symbol $\widetilde{T}$ to denote the integrated form of $T$ (see \cite[\S VIII.12]{MR936629}). But sometimes we directly regard $T$ as $\ast$-representation of $\mathscr{L}_2(G, \mathscr{B})$ or $C^{\ast}(\mathscr{B})$, i.e its integrated form. 

In this section we review some basic notions of the theory on induced representations of $C^{\ast}$-algebraic Fell bundles which are treated in details in \cite[\S XI.9]{MR936629}. 

Let $S$ be a $\mathscr{B}$-$positive$ (see \cite[\S XI.8]{MR936629}) non-degenerate $\ast$-representation of $\mathscr{B}_H$.  Let $Z_{\alpha}$ be the algebraic direct sum $\sum_{x \in \alpha}^{\oplus}(B_x \otimes X(S))$ of the algebraic tensor products $B_x \otimes X(S)$, we introduce into $Z_{\alpha}$ the conjugate-bilinear form $( \ ,\ )_{\alpha}$ by
\begin{equation*}
(b \otimes \xi, c \otimes \eta)_{\alpha}=(\rho(x)\rho(y))^{-1/2}(S_{c^{\ast}b}\xi, \eta)
\end{equation*}
($x, \ y \in \alpha; \ b \in B_x; \ c \in B_y; \ \xi, \ \eta \in X(S)$). One can form a Hilbert space $Y_{\alpha}$ by factoring out from $Z_{\alpha}$ the null space of $( \ , \ )_{\alpha}$ and completing, and a Hilbert bundle $\mathscr{Y}$ over $G/H$ with fibers $Y_{\alpha}$. Let $\kappa_{\alpha}: Z_{\alpha} \to Y_{\alpha}$ be the quotient map for each $\alpha \in G/H$. For each $c \in B_x$, there is continuous map $\tau_{c}: Y \to Y$ defined by
\begin{equation*}
\tau_c(\kappa_{\alpha}(b \otimes \xi))=\kappa_{x \alpha}(cb \otimes \xi) \ \ \ \ (b \in B_{y}, y \in \alpha; \xi \in X(S)).
\end{equation*}
Furthermore for each $b \in B$ 
\begin{equation*}
(\alpha \mapsto f(\alpha)) \mapsto (f: \alpha \mapsto \tau_bf(x^{-1}\alpha)) \ \ \ \ (f \in \mathscr{L}_2(G/H; \mathscr{Y}))
\end{equation*}
is a bounded operator on the Hilbert space $\mathscr{L}_2(G/H; \mathscr{Y})$, which we denote by $T_b$, and that $b \mapsto T_b$ is non-degenerate $\ast$-representation of $\mathscr{B}$. We denote this representation $T$ by ${\rm{Ind}}_{\mathscr{B}_H \uparrow \mathscr{B}}(S)$, and say that $T$ is $induced$ from $S$.

By \cite[XI.11.7]{MR936629} every $\ast$-representation $S$ of $B_e$ is $\mathscr{B}$-positive; if in addition $\mathscr{B}$ is saturated, by \cite[XI.11.10]{MR936629} every $\ast$-representation $S$ of $\mathscr{B}_H$ is $\mathscr{B}$-positive. Therefore, in either of the case $H=\{e\}$ or that $\mathscr{B}$ is saturated, ${\rm{Ind}}_{\mathscr{B}_H \uparrow \mathscr{B}}(S)$ exists.

Let $(G,M)$ be a $G$-transformation group. By \cite[\S VIII.7]{MR936629} one can construct the $G, M$ $transformation$ $bundle$ $\mathscr{D}$ $over$ $G$ $derived$ $from$ $\mathscr{B}$, whose fiber $D_x$ for each $x \in G$ is $\mathscr{C}_0(M; B_x)$, i.e the set of continuous functions from $M$ into $B_x$ vanishing at infinity point, with multiplication and involution given by
\begin{equation*}\begin{split}
&(\phi\psi)(m)=\phi(m)\psi(x^{-1}m),
\\& \phi^{\ast}(m)=(\phi(xm))^{\ast}
\end{split}\end{equation*}
($x,y \in G; \phi \in D_x; \psi \in D_y; m \in M$). $A$ $system$ $of$ $imprimitivity$ $for$ $\mathscr{B}$ $over$ $M$ is a pair $\langle T, P \rangle$ where: (i)  $T$ is a non-degenerate $\ast$-representation $T$ of $\mathscr{B}$; (ii) $P$ is a regular $X(T)$-projection-valued Borel measure on $M$; and (iii) we have
\begin{equation}\label{vnnhjordniijnc}
T_bP(W)=P(\pi(b)W)T_b
\end{equation}
for all $b \in B$ and all Borel subsets $W$ of $G/H$. One can show that for any system of imprimitivity  $\langle T, P \rangle$ 
\begin{equation*}
\langle T, P \rangle'_{\phi}=\int_{M}dPm T_{\phi(m)} \ \ \ \ (\phi \in D)
\end{equation*}
is a $\ast$-representation of $\mathscr{D}$, and that $\langle T, P \rangle \mapsto \langle T, P \rangle'$ is a one-to-one correspondence between the unitary classes of systems of imprimitivity and unitary classes of non-degenerate $\ast$-representations of $\mathscr{D}$. For the details and the definition of the integration appearing in (\ref{vnnhjordniijnc}) we refer the reader to \cite[\S VIII.18]{MR936629}. In this paper we will always identify the systems of imprimitivity and non-degenerate $\ast$-representations of $\mathscr{D}$. In the rest of this paper, we assume that we have implicitly chosen a transformation group $(G,M)$ and use the symbol $\mathscr{D}$ to denote the corresponding transformation bundle.

For $M=G/H$ with the left translated action, the transformation bundle plays very important role. Let us provide details. For a $\mathscr{B}$-positive non-degenerate $\ast$-representation $S$ of  $\mathscr{B}_H$, let $\mathscr{Y}$ be the Hilbert bundle over $G/H$ constructed as above, we define a $\mathscr{L}_2(G/H; \mathscr{Y})$-projection valued measure $P$ on $G/H$ by
\begin{equation*}
P(W)f={\rm{Ch}}_Wf \ \ \ \ (f \in \mathscr{L}_2(G/H; \mathscr{Y}))
\end{equation*}
(${\rm{Ch}}_W$ is the characteristic function) for Borel subsets $W \subset G/H$. Let $T={\rm{Ind}}_{\mathscr{B}_H \uparrow \mathscr{B}}(S)$. It is easy to show that $\langle T, P \rangle$ is a system of imprimitivity. If $\mathscr{B}$ is saturated, by \cite[XI.14.18]{MR936629} the correspondence $S \mapsto \langle T, S \rangle$ (up to unitary equivalent classes) is actually Rieffel's inducing process with respect to the pre-$C^{\ast}(\mathscr{B}_H)$-$C^{\ast}(\mathscr{D})$ Hilbert bimodule $\mathscr{L}(\mathscr{B})$. We can conclude that $C^{\ast}(\mathscr{B}_H)$ and $C^{\ast}(\mathscr{D})$ are Morita equivalent. In the case that $G$ is second countable and $\mathscr{B}$ is separable, this theorem is proved by Kaliszewski, Muhly, Quigg and Williams  \cite{ijfjkfw} as an easy consequence of the theory of equivalence of Fell bundles over locally compact groupoids. 

For the details on the definition of the regular representation of $\mathscr{B}$ we refer the readers to Exel and Ng \cite{ MR1891686}. $\mathscr{B}$ is said to be $amenable$ if the regular representation, regarded as a representation of the full $C^{\ast}$-algebra $C^{\ast}(\mathscr{B})$, is faithful. In the same paper, the Fell's Absorption Theorem of the bundle version is proved which we will use frequently: If $T$ is a non-degenerate $\ast$-representation of $\mathscr{B}$ such that $T|B_e$ is faithful, then inner tensor product $T \otimes R^G$ is weakly equivalent to the regular representation of $\mathscr{B}$, where $R^G$ is the left-regular representation of $G$.

\section{A Remark on Fell-Doran's Imprimitivity Theorem}

In this section, we assume that $M=G/H$.

Let $\mathscr{C}$ and $\mathscr{C}'$ be two Banach bundles  with fiber spaces $C$ and $C'$ over the same locally compact space $M'$ with Borel measure $\mu$. If $F: C \to C'$ is a continuous map satisfying: (1) $F|C_x$ is bounded linear map from $C_x$ into $C_x'$; (2) there is constant $K>0$ such that $\|F|C_x\| \leq K$ for all $x \in G$ we say that $F$ is a $\mathscr{C}$-$\mathscr{C}'$ $multiplier$ $of$ $order$ $e$, and we use symbol $\widetilde{F}$ to denote the map from $\mathscr{L}_2(\mathscr{C})$ into $\mathscr{L}_2(\mathscr{C}')$ defined by
\begin{equation*}
\widetilde{F}(f)(x)=F(f(x)) \ \ \ \ (f \in \mathscr{L}_2(\mathscr{C}), x \in M).
\end{equation*}
Furthermore, if $F$ is norm-preserving and bijective, it is routine to verify that $\widetilde{F}$ is unitary, and we say that $F$ is $\mathscr{C}$-$\mathscr{C}'$ $unitary$ $multiplier$ $of$ $order$ $e$.

\begin{lemma}\label{cbnujhdfliilds}
Let $S$ be a non-degenerate $\ast$-representation of $\mathscr{B}_H$, $U$ a non-degenerate unitary representation of $G$. Let $\langle T,P \rangle$ be the system of imprimitivity induced from $S$, then the system of imprimitivity induced from the inner tensor product $S \otimes (U|H)$ is unitarily equivalent to $(T \otimes U, P \otimes 1_{\mathcal{O}(X(U))})$.
\end{lemma}
\begin{proof}
It is easy to verify that the Hilbert bundle over $(G \times G)/(H \times G)$ induced by the outer tensor product $S \underset{o}{\otimes} U$ is $\mathscr{Y}'=\{Y'_{\alpha}=Y_{\alpha} \otimes X(U)\}_{\alpha \in G/H}$ over $G/H$ by identifying $(G \times G)/(H \times G)$ with $G/H$. Then we have a unitary operator $E: X(U) \to X(U)$ such that
\begin{equation}\begin{split}\label{uiehrybhfdu}
& \ \ \ \ (1_{\mathcal{O}(\mathscr{L}_2(\rho^{\sharp}; \mathscr{Y}))} \otimes E^{\ast}) {\rm{Ind}}_{\mathscr{B}_H \times G \uparrow \mathscr{B} \times G}(S \underset{o}{\otimes}U)(1_{\mathcal{O}(\mathscr{L}_2(\rho^{\sharp}; \mathscr{Y}))} \otimes E)
\\&={\rm{Ind}}_{\mathscr{B}_H \uparrow \mathscr{B}}(S) \underset{o}{\otimes} U.
\end{split}\end{equation}
On the other hand, let $\mathscr{Z}=\langle Z, Z_{\alpha} \rangle$ be the Hilbert bundle over $G/H$ induced from $S \otimes (U|H)$, by the proof of \cite[XI.13.2]{MR936629} there is a $\mathscr{Z}$-$\mathscr{Y}'$ unitary multiplier of order $e$ $F: \mathscr{Z} \to \mathscr{Y}'$ such that 
\begin{equation}\label{xhgidfsui}
\widetilde{F} {\rm{Ind}}_{(\mathscr{B}_H \times G)_{D_H} \uparrow (\mathscr{B} \times G)_{D_G}}(S \underset{o}{\otimes}U|D_H)\widetilde{F}^{\ast}={\rm{Ind}}_{\mathscr{B}_H \times G \uparrow \mathscr{B} \times G}(S \underset{o}{\otimes}U)|D_G,
\end{equation}
where $D_G=\{\langle x, x \rangle : x \in G\}$, $D_H=\{\langle x, x \rangle : x \in H\}$. Let  $J: \mathscr{Z} \to \mathscr{Y}'$ be defined by $J=(1_{\mathscr{Y}} \otimes E^{\ast}) \circ F$, where $1_{\mathscr{Y}}$ is the identity map from $\mathscr{Y}$ to itself, then $J$ is unitary $\mathscr{Z}$-$\mathscr{Y}'$ multiplier and by (\ref{uiehrybhfdu}) and (\ref{xhgidfsui}) we have
\begin{equation*}
\widetilde{J} {\rm{Ind}}_{(\mathscr{B}_H \times G)_{D_H} \uparrow (\mathscr{B} \times G)_{D_G}}(S \underset{o}{\otimes}U|D_H)\widetilde{J}^{\ast}=({\rm{Ind}}_{\mathscr{B}_H \uparrow \mathscr{B}}(S) \underset{o}{\otimes} U)|D,
\end{equation*}
hence we have
\begin{equation*}
\widetilde{J}{\rm{Ind}}_{\mathscr{B}_H \uparrow \mathscr{B}}(S \otimes (U|H))\widetilde{J}^{\ast}={\rm{Ind}}_{\mathscr{B}_H \uparrow \mathscr{B}}(S) \otimes U.
\end{equation*}
Let $\langle {\rm{Ind}}_{\mathscr{B}_H \uparrow \mathscr{B}}(S \otimes (U|H)), Q \rangle$ be the system of imprimitivity induced from $S \otimes (U|H)$, it is easy to see that $\widetilde{J} Q \widetilde{J}^{\ast}=P \otimes 1_{\mathcal{O}(X)}$, thus we have
\begin{equation*}
\widetilde{J} \langle {\rm{Ind}}_{\mathscr{B}_H \uparrow \mathscr{B}}(S \otimes (U|H)), Q \rangle \widetilde{J}^{\ast}=\langle T \otimes U, P \otimes 1_{\mathcal{O}(X)} \rangle,
\end{equation*}
 our proof is done. 
\end{proof}

The following lemma is well-known but we did not find reference. Our proof might be easier than the standard proof:

\begin{lemma}\label{cnbjhlsdlilfr}
Let $\mathscr{B}$ be the group bundle $\mathbb{C} \times G$, then the $\ast$-representation of $\mathscr{D}$ which is  induced from $R^H$ is weakly equivalent to the regular representation of $\mathscr{D}$.
\end{lemma}
\begin{proof}
Let $\langle T, P \rangle$ be the system of imprimitivity induced by $R^H$. Then by Lemma \ref{cbnujhdfliilds} the system of imprimitivity induced by $R^H \otimes (R^G|H)$ is $\langle T \otimes R^G, P \otimes 1_{\mathscr{L}_2(G)} \rangle$. Let $Q$ and $Q'$ be the integrated forms of $\langle T, P \rangle$ and $\langle T \otimes R^G,  P \otimes 1_{\mathscr{L}_2(G)} \rangle$ respectively, then we have
\begin{equation*}
Q'=Q \otimes R^G.
\end{equation*}
Since $Q|D_e$ is faithful, by the Fell's Absorption Theorem we conclude that $Q'$ is weakly equivalent to regular representation of $\mathscr{D}$. On the other hand, by Fell's Absorption Theorem again $R^H \otimes (R^G|H)$ is weakly equivalent to $R^H$, hence $Q$ is weakly equivalent to $Q'$, and so $Q$ is weakly equivalent to regular representation of $\mathscr{D}$, our proof is done. 
\end{proof}

The following lemma is well known, e.g. see Echterhoff and Raeburn \cite{sjjfdjncj}. Our proof based on Lemma \ref{cbnujhdfliilds} is easier:
\begin{lemma}\label{xmhvfurfsd}
$R^G|H$ is weakly equivalent to $R^H$.
\end{lemma}
\begin{proof}
Let $\langle T, P \rangle$ be the system of imprimitivity induced by the trivial representation $S$ of $H$. By Lemma \ref{cbnujhdfliilds} the system of imprimitivity induced by $S \otimes (R^G|H)$ is $\langle T \otimes R^G, P \otimes 1_{\mathcal{O}(\mathscr{L}_2(G))} \rangle$, then by the same argument of the proof of Lemma \ref{cnbjhlsdlilfr} $\langle T \otimes R^G, P \otimes 1_{\mathcal{O}(\mathscr{L}_2(G))} \rangle$ is weakly equivalent to the regular representation of $\mathscr{D}$. Thus by  Lemma \ref{cnbjhlsdlilfr} $S \otimes (R^G|H)$ is weakly equivalent to $R^H$, but $S \otimes (R^G|H)=R^G|H$, the proof is done.
\end{proof}

The following corollary improve Echterhoff and Quigg \cite[Proposition 6.3]{rjnkllds} in which $G$ is discrete and $\mathscr{B}$ has AP:
\begin{corollary}\label{kjdsnf}
Assume $\mathscr{B}$ is saturated and amenable. If for any non-degenerate $\ast$-representation of $\mathscr{B}_H$ we have ${\rm{Ind}}_{\mathscr{B}_H \uparrow \mathscr{B}}(S)|H \geq S$ , then $\mathscr{B}_H$ is amenable. In particular, either if $H$ is normal closed subgroup or $G/H$ is discrete, $\mathscr{B}_H$ is amenable.
\end{corollary}
\begin{proof}
Let $S$ be a faithful non-degenerate $\ast$-representation of $C^{\ast}(\mathscr{B}_H)$. By Lemma \ref{xmhvfurfsd} and \cite[XI.13.2]{MR936629} 
\begin{equation*}
{\rm{Ind}}_{\mathscr{B}_H \uparrow \mathscr{B}}(S \otimes (R^G|H))|\mathscr{B}_H={\rm{Ind}}_{\mathscr{B}_H \uparrow \mathscr{B}}(S)|\mathscr{B}_H \otimes (R^G|H)
\end{equation*}
is weakly equivalent to regular representation of $\mathscr{B}_H$. On the other hand, since $\mathscr{B}$ is amenable, ${\rm{Ind}}_{\mathscr{B}_H \uparrow \mathscr{B}}(S \otimes (R^G|H))$ is faithful $\ast$-representation of $C^{\ast}(\mathscr{B})$, thus we have 
\begin{equation*}
{\rm{Ind}}_{\mathscr{B}_H \uparrow \mathscr{B}}(S)|\mathscr{B}_H \otimes (R^G|H) \geq {\rm{Ind}}_{\mathscr{B}_H \uparrow \mathscr{B}}(S)|\mathscr{B}_H \geq S,
\end{equation*}
thus $\mathscr{B}_H$ is amenable.

The other parts of the proof is completed by the combination of \cite[XI.14.21]{MR936629} and \cite[XI.12.8]{MR936629}, which state that either if $G/H$ is discrete or $H$ is normal closed subgroup then for any non-degenerate $\ast$-representation $S$ of $\mathscr{B}_H$ we have $S \leq {\rm{Ind}}_{\mathscr{B}_H \uparrow \mathscr{B}}(S)|H$.
\end{proof}

\begin{theorem}\label{vnkjlslinvrfd}
Assume that $\mathscr{B}$ is saturated. Let $\langle T, P \rangle$ be the system of imprimitivity induced by $S$ which is non-degenerate $\ast$-representation of $\mathscr{B}_H$, then $\langle T, P \rangle$ is weakly equivalent to regular representation of $\mathscr{D}$ if and only if $S$ is weakly equivalent to regular representation of $\mathscr{B}_H$.
\end{theorem}
\begin{proof}
Let $\langle T', P' \rangle$ be faithful $\ast$-representation of $C^{\ast}(\mathscr{D})$, $S'$ be the $\ast$-representation of $\mathscr{B}_H$ inducing $\langle T', P' \rangle$, and $Q'$ be the integrated form of $\langle T', P' \rangle$. Then $S'$ is faithful $\ast$-representation of $C^{\ast}(\mathscr{B}_H)$, in particular $S'|B_e$ is faithful, then by Lemma \ref{xmhvfurfsd} $S' \otimes (R^G|H)$ is weakly equivalent to regular representation of $\mathscr{B}_H$. On the other hand, by Lemma \ref{cbnujhdfliilds} $\langle Q' \otimes R^G, P' \otimes 1_{\mathscr{L}_2(G)}\rangle$ is induced by $S' \otimes (R^G|H)$, and since $\langle Q' \otimes R^G, P' \otimes 1_{\mathscr{L}_2(G)}\rangle$ is weakly equivalent to the regular representation of $\mathscr{D}$, our proof is completed.
\end{proof}

\section{Tensor Products of $C^{\ast}$-Algebraic Bundles and $C^{\ast}$-Algebras}
In this section we study the tensor product of Fell bundles and $C^{\ast}$-algebras. Our method is not hard to generalize to construct tensor products of $C^{\ast}$-algebras which will be treated in a forthcoming paper \cite{dhhldsw} by the present author. For our goal of this paper we confine our attention on this specific case. 

Let us make some general convention. Let $E$ be a $\ast$-algebra. If $A$ is a $C^{\ast}$-algebra such that there is $\ast$-homomorphism $r_A: E \to A$ such that $r(E)$ is norm-dense in $A$, then we say that $A$ is $\ast$-$quotient$ of $E$, or $A$ is $quotient$ $C^{\ast}$-$algebra$ of $E$. If $B$ is another quotient $C^{\ast}$-algebra of $E$ and $\|r_A(c)\|=\|r_B(c)\|$ for all $c \in E$, then it is easy to see that $r_B(c) \mapsto r_A(c)$ can be extended to faithful $\ast$-homomorphism from $B$ onto $A$. In this case we say that $A$ and $B$ $have$ $same$ $\ast$-$source$, thus if for any $C^{\ast}$-algebras $A$ and $B$ we can prove that they have same $\ast$-source, then we have proved that they are $\ast$-isomorphic.

In this section let $A$ be a fixed $C^{\ast}$-algebra, let $B_x \otimes A$ denote the algebraic tensor product of $B_x$ and $A$ for each $x \in G$. Then we can form a $\ast$-algebraic bundle $\mathscr{B}^d \otimes A=\{B_x \otimes A: x \in G\}$ over $G$ by defining
\begin{equation*}\begin{split}
&(\sum_{i=1}^n b_i \otimes a_i)^{\ast}=\sum_{i=1}^n b_i ^{\ast} \otimes a_i ^{\ast} \ \ \ \ (b_i \in B_x, x \in G; a_i \in A),
\\& (\sum_{i=1}^n b_i \otimes a_i)(\sum_{j=1}^m b'_j \otimes a'_j)=\sum_{i=1}^n\sum_{j=1}^m b_ib'_j \otimes a_i a'_j
\end{split}\end{equation*}  
$(b_i \in B_x, b'_j \in B_y, x,y \in G; a_i, a'_j \in A$. To see these are well defined, we just need to regard $\{B_x \otimes A: x \in G\}$ as subset of algebraic tensor product $C^{\ast}(\mathscr{B}^d) \otimes A$.

In the rest of this section we denote the linear span of $\{x \mapsto f(x)\otimes a: f \in \mathscr{L}(\mathscr{B}), a \in A\}$ by $\Gamma$. 

 For any pre-$C^{\ast}$-seminorm $r$ of $\mathscr{B}^d \otimes A$ we let $B_x \otimes_{r} A$ be the completion of the quotient of $B_x \otimes A$ with respect to the seminorm $r|(B_x \otimes A)$. By Lemma \ref{cnmbjhkdfuuks} and \cite [II.13.18]{MR936628} we can define a $C^{\ast}$-algebraic bundle over $G$ with fibers $B_x \otimes_{r} A$ such that all the members of $\Gamma$ is continuous cross-sections. We denote this $C^{\ast}$-algebraic bundle by $\mathscr{B} \otimes_{r} A$. Notice that if $\mathscr{B} \otimes A$ have pre-$C^{\ast}$-seminorms $r_1$ and $r_2$ such that $r_1|(B_e \otimes A)=r_2|(B_e \otimes A)$, then $r_1=r_2$ and $\mathscr{B} \otimes_{r_1}A=\mathscr{B}\otimes_{r_2}A$.

Therefore $\Gamma$ is norm dense in $\mathscr{L}_1(\mathscr{B} \otimes_r A)$, and $C^{\ast}(\mathscr{B}\otimes_r A)$ is $\ast$-quotient of $\Gamma$.

The proof of the following lemma is routine: 
\begin{lemma}\label{vliedkdsasd}
Let $T$ be a non-degenerate $^\ast$-representation of the $\ast$-algebra $C^{\ast}(\mathscr{B}^d) \otimes A$. By regarding $\{B_x \otimes A\}_{x \in G} \subset C^{\ast}(\mathscr{B}) \otimes A$, for each $x \in G$ we define semi-norm $\sigma_x$ on $B_x \otimes A$ by $T|(B_x \otimes A)$, then $\sigma=\cup_{x \in G}\sigma_x$ is a pre-$C^{\ast}$ seminorm of $\mathscr{B}\otimes A$.
\end{lemma}

\begin{lemma}\label{cnmbjhkdfuuks}
Let $\sigma$ be a $C^{\ast}$-norm of the $\ast$-algebra $B_e \otimes A$, then there is a pre-$C^{\ast}$-seminorm $\sigma'$ on $\{B_x \otimes A: x \in G\}$ such that $\sigma(c) \leq (\sigma'|B_e \otimes A)(c)$ for all $c \in B_e \otimes A$.
\end{lemma}
\begin{proof}
Let $\langle S, R \rangle$ be faithful $\ast$-representation of $B_e \otimes_{\sigma} A$. We define 
\begin{equation}
T=\rm{Ind}_{B_e \uparrow \mathscr{B}^d}(S).
\end{equation} 
Let $\mathscr{Y}^d$ be the Hilbert bundle over $G^d$ which is induced by $S$, and let $R^0_a$ be an operator defined on $\mathscr{L}_2(G^d; \mathscr{Y}^d)$ by 
\begin{equation}
R^0_a(\kappa_x(b \otimes \xi))=\kappa_x(b \otimes R_a(\xi)) \ \ \ \ (x \in G, b \in B_x, \xi \in X(\langle S, R \rangle)).
\end{equation}
Then $R_a^0$ is in the commuting algebra of $T$. Furthermore, let $F: Y_e \to X(S)$ be the unitary operator defined by
\begin{equation}
F(\kappa_e(b \otimes \xi))=S_b\xi \ \ \ \ (b \in B_e, \xi \in X(S)),
\end{equation} 
we have
\begin{equation}\begin{split}
FR^0_a(\kappa_e (b \otimes \xi))&=F(\kappa_e(b \otimes R_a(\xi)))
\\&=S_b(R_a(\xi))
\\&=R_a S_b(\xi)
\\&=R_aF(\kappa_e(b \otimes \xi))
\end{split}\end{equation}
$(b \in B_e, \xi \in X(\langle S, R \rangle))$, hence
\begin{equation}
R=F^{\ast}(^{Y_e}R^0)F.
\end{equation}
On the other hand, by \cite[XI.14.21]{MR936629}
\begin{equation}
F^{\ast}(^{Y_e}(T|B_e))F=S,
\end{equation}
we conclude that if $\sigma'$ is defined by
\begin{equation}
\sigma'(\sum_{i=1}^n b_i \otimes a_i)=\|\sum_{i=1}^n T_{b_i} R^0_{a_i}\| \ \ \ \ (\sum_{i=1}^n b_i \otimes a_i \in B_x \otimes A, x \in G),
\end{equation}
then $\sigma \leq \sigma'| (B_e \otimes A)$. Finally, by Lemma \ref{cnmbjhkdfuuks} $\sigma'$ is a pre-$C^{\ast}$-seminorm on $\{B_x \otimes A: x \in G\}$.
\end{proof}

Now we have the following important proposition:
\begin{proposition}
 The maximal and minimal norms of $B_e \otimes A$ can be extended to unique pre-$C^{\ast}$-seminorms of $\mathscr{B}^d \otimes A$. We denote the corresponding $C^{\ast}$-algebraic bundles over $G$ by $\mathscr{B} \otimes_{max}A$ and $\mathscr{B} \otimes_{min}A$.
\end{proposition}
\begin{proof}
If $\sigma$ is the minimal norm of $B_e \otimes A$, it is easy to see that the $\sigma'$ constructed in the proof of Lemma \ref{cnmbjhkdfuuks} satisfies $\sigma'|(B_e \otimes A)=\sigma$. Furthermore, by \cite[XI.11.3]{MR936629} if $\sigma$ is the maximal tensor norm of $B_e \otimes A$, then $\sigma'|(B_e \otimes A)=\sigma$. 
\end{proof}

The following lemma is readily proved according to definitions:
\begin{lemma}\label{cnuuod}
The map
\begin{equation}
\Psi_{\sigma}: \sum_{i=1}^n b_i \otimes a_i \mapsto \sum_{i=1}^n b_i \otimes a_i \ \ \ \ (\sum_{i=1}^n b_i \otimes a_i \in B_x \otimes A, x \in G)
\end{equation}
 can be extended to a continuous map from $\mathscr{B} \otimes_{\rm{max}} A$ onto $\mathscr{B} \otimes_{r} A$. Therefore, if $T$ is a $\ast$-representation of $\mathscr{B} \otimes_{r} A$, then $T \circ \Psi$ is a $\ast$-representation of $\mathscr{B} \otimes_{\rm{max}} A$.
\end{lemma}

\begin{lemma}\label{cxmhhj}
For each $b \in B_x$, $b_i \in B_y$ and $a_i \in A$ ($x, y \in G; i=1,...n$), we have
\begin{equation}
\|\sum_{i=1}^n (b_ib)\otimes a_i\|_{\mathscr{B} \otimes_{\rm{max}}A}, \|\sum_{i=1}^n (bb_i)\otimes a_i\|_{\mathscr{B} \otimes_{\rm{max}}A} \leq \|b\| \|\sum_{i=1}^n b_i\otimes a_i\|_{\mathscr{B} \otimes_{\rm{max}}A}
\end{equation}
and
\begin{equation}
\|\sum_{i=1}^n (b_i)\otimes (a_ia)\|_{\mathscr{B} \otimes_{\rm{max}}A}, \|\sum_{i=1}^n (b_i)\otimes (aa_i)\|_{\mathscr{B} \otimes_{\rm{max}}A} \leq \|a\| \|\sum_{i=1}^n b_i\otimes a_i\|_{\mathscr{B} \otimes_{\rm{max}}A}
\end{equation}
\end{lemma}
\begin{proof}
We prove 
\begin{equation}
\|\sum_{i=1}^n (b_ib)\otimes a_i\|_{\mathscr{B} \otimes_{\rm{max}}A} \leq \|b\| \|\sum_{i=1}^n b_i\otimes a_i\|_{\mathscr{B} \otimes_{\rm{max}}A},
\end{equation}
the other parts may be proved by the same argument.

By the proof of Lemma \ref{cnmbjhkdfuuks}, we have $\ast$-representation $T$ of $\mathscr{B}^d$ and $\ast$-representation $S$ of $A$ such that
\begin{equation}
\|\sum_{i=1}^n b_i\otimes a\|_{\mathscr{B} \otimes_{\rm{max}}A}=\|\sum_{i=1}^n T_{b_i}S_{a_i}\| \ \ \ \ (\sum_{i=1}^n b_i \otimes a_i \in B_z \otimes A, z \in G).
\end{equation}
Thus we have
\begin{equation}\begin{split}
\|\sum_{i=1}^n(b_ib)\otimes a_i\|_{\mathscr{B} \otimes_{\rm{max}}A}&=\|\sum_{i=1}^nT_{b_ib}S_{a_i}\|
\\&=\|T_b\| \|\sum_{i=1}^nT_{b_i}S_{a_i}\|
\\&\leq \|b\| \|\sum_{i=1}^n b_i \otimes a_i\|_{\mathscr{B} \otimes_{\rm{max}}A}.
\end{split}\end{equation}
Our proof is done.
\end{proof}

By the previous lemma, for each $b \in B$ and $a \in A$, we can define
\begin{equation*}\begin{split}
&_bu:  \mathscr{B} \otimes_{\rm{max}}A  \to \mathscr{B} \otimes_{\rm{max}}A, \ \sum_{i=1}^n b_i \otimes a_i \mapsto \sum_{i=1}^n (bb_i) \otimes a_i
\\&u_b: \mathscr{B} \otimes_{\rm{max}}A  \to \mathscr{B} \otimes_{\rm{max}}A, \ \sum_{i=1}^n b_i \otimes a_i \mapsto \sum_{i=1}^n (b_ib) \otimes a_i;
\\& _av: \mathscr{B} \otimes_{\rm{max}}A  \to \mathscr{B} \otimes_{\rm{max}}A, \ \sum_{i=1}^n b_i \otimes a_i \mapsto \sum_{i=1}^n b_i \otimes (aa_i)
\\& v_a:  \mathscr{B} \otimes_{\rm{max}}A  \to \mathscr{B} \otimes_{\rm{max}}A, \ \sum_{i=1}^n b_i \otimes a_i \mapsto \sum_{i=1}^n b_i \otimes (a_ia).
\end{split}\end{equation*}

\begin{lemma}\label{xchjjiooifdsd}
$_bu$, $u_b$, $_av$ and $v_a$ are continuous. In particular, $\langle _bu, u_b \rangle$ and $\langle _av, v_a \rangle$ are multipliers of $\mathscr{B} \otimes_{\rm{max}}A$ of order $\pi(b)$ and $e$ respectively.
\end{lemma}
\begin{proof}
We prove the continuity of $u_b$, the continuity of the others can be proved by the same argument.

For any $f \in \Gamma$, it is easy to see that $x \mapsto u_bf(x)$ is continuous, for this is the consequence of the following: If $b_i \to b$ in $B$ then $b_i \otimes a \to b \otimes a$ in $\mathscr{B} \otimes_{\rm{max}}A$.  To prove this, let $g \in \mathscr{L}(\mathscr{B})$ such that $g(\pi(b))=b$, then $\|g(\pi(b_i))-b_i\| \to 0$ and so 
\begin{equation}
\|g(\pi(b_i))\otimes a-b_i \otimes a\| \leq \|g(\pi(b_i))-b_i\|\|a\| \to 0.
\end{equation}
On the other hand $g(\pi(b_i)) \otimes a \to g(\pi(b)) \otimes a=b \otimes a$, by \cite[III.13.12]{MR936628} we conclude that $b_i \otimes a \to b \otimes a$. Therefore we have proved the continuity of $x \mapsto u_bf(x)$.

Let $\{c_i\}_{i \in I} \subset B \otimes_{\rm{max}}A$ such that $c_i \to c$. For arbitrary $\epsilon >0$, it is easy to see that there is $f \in \Gamma$ with $\|f(\pi(c))-c\| < \epsilon$. Then $\|f(\pi(c_i))-c_i\| < \epsilon$ for large $i$. Thus by Lemma \ref{cxmhhj} we have
\begin{equation*}\begin{split}
\|u_bf(\pi(c_i))-u_bc_i\| &< \|b\| \epsilon.
\end{split}\end{equation*}
On the other hand, we have proved that $u_bf(\pi(c_i)) \to u_bf(\pi(c))$ and
\begin{equation*}
\|u_bf(\pi(c))-u_bc\| \leq \|b\| \|f(\pi(c))-c\| < \|b\| \epsilon,
\end{equation*}
by \cite[III.13.12]{MR936628} again we have $u_bc_i \to u_b c$, this proved the continuity of $u_b$.
\end{proof}

\begin{lemma}\label{cnuokudfskhfs}
$b \mapsto u_b$ and $a \mapsto v_a$ are strongly continuous.
\end{lemma}
\begin{proof}
Let $\{b_i\}$ be a net of $B$ converging to $b \in B$. By Lemma \ref{xchjjiooifdsd}, we have
\begin{equation*}
u_{b_i}(\sum_{k=1}^n b_k \otimes a_k) \to \sum_{k=1}^n bb_k \otimes a_k \ \ \ \ (b_k \in B_x, x \in G; a_k \in A).
\end{equation*}
Now for any $c \in B \otimes_{\rm{max}} A$ with $\pi(b)=B_x$, for arbitrary $\epsilon >0$ there are $b_k \in B_x$ and $a_k \in A$ such that 
\begin{equation*}
\|\sum_{i=1}^n b_k \otimes a_k -c\| < \epsilon.
\end{equation*}
Thus by  Lemma \ref{cxmhhj} we have
\begin{equation*}
\|u_{b_i}(\sum_{i=1}^n b_k \otimes a_k-c)\|< \|b_i\| \cdot \epsilon.
\end{equation*}
On the other hand, 
\begin{equation*}
u_{b_i}(\sum_{k=1}^n b_k \otimes a_k) \to u_b(\sum_{k=1}^n b_k \otimes a_k),
\end{equation*}
 by  \cite[III.13.12]{MR936628} we have $u_{b_i}c \to u_bc$.
\end{proof}

By Lemma \ref{cnuuod}, Lemma \ref{xchjjiooifdsd}, Lemma \ref{cnuokudfskhfs} and \cite[VIII.15.3]{MR936629} we conclude the following proposition:
\begin{proposition}\label{cnhjousd}
For any non-degenerate $\ast$-representation $T$ of $\mathscr{B} \otimes_{r} A$, there are $\ast$-representations of $\mathscr{B}$ and $A$, say $S$ and $R$, such that $range(S)$ is in the commuting algebra of $R$ and
\begin{equation*}
T(\sum_{i=1}^n b_i \otimes a_i)=\sum_{i=1}^n S_{b_i}R_{a_i} \ \ \ \ (\sum_{i=1}^n b_i \otimes a_i \in B_x \otimes A, x \in G).
\end{equation*}
\end{proposition}

\begin{proposition}\label{vnjilfrf}
Every non-degenerate $\ast$-representation of $C^{\ast}(\mathscr{B}) \otimes_{max} A$ is the integrated form of a unique non-degenerate $\ast$-representation of $\mathscr{B} \otimes_{max} A$.
\end{proposition}
\begin{proof}
Let $\langle \widetilde{S}, R \rangle$ be a $\ast$-representation of $C^{\ast}(\mathscr{B}) \otimes_{max} A$, let $S$ be $\ast$-representation of $\mathscr{B}$ such that $\widetilde{S}$ is the integrated form of $S$. It is easy to verify that 
\begin{equation*}
\|\sum_{i=1}^nS_{b_i}R_{a_i}\|=\|\sum_{i=1}^n b_i \otimes a_i\|_{max} \ \ \ \ (b_i \in B_e, a_i \in A), 
\end{equation*}
thus the map
\begin{equation*}
T: \sum_{i=1}^n b_i \otimes a_i \mapsto \sum_{i=1}^n S_{b_i} R_{a_i} \ \ \ \ (b_i \in B_x, x \in G; a_i \in A),
\end{equation*}
can be extended to a $\ast$-representation of $(\mathscr{B} \otimes_{max} A)^d$. Furthermore, for each $g \in \Gamma$ the map
\begin{equation*}
x \mapsto T_{g(x)}
\end{equation*}
is strongly continuous, thus we conclude that $T$ can be extended to $\ast$-representation of $\mathscr{B} \otimes_{max}A$ whose integrated form is $\langle \widetilde{S}, R \rangle$.
\end{proof}

By Proposition \ref{vnjilfrf} and Proposition \ref{cnhjousd} together, we conclude that $C^{\ast}(\mathscr{B} \otimes_{max}A)$ and $C^{\ast}(\mathscr{B}) \otimes_{max}A$ have same $\ast$-source, thus we have proved the first part of the following proposition:

\begin{proposition}\label{cjjfifshlfgslifs}
For any $C^{\ast}$-algebra $A$, 
\begin{equation}
C^{\ast}(\mathscr{B} \otimes_{max}A)=C^{\ast}(\mathscr{B}) \otimes_{max}A,
\end{equation}
\begin{equation*}
C_r^{\ast}(\mathscr{B} \otimes_{min}A)=C_r^{\ast}(\mathscr{B}) \otimes_{min}A
\end{equation*} 
\end{proposition}
\begin{proof}
Let us see the proof of the second part.  Let $T$ be faithful $\ast$-representation of $\mathscr{B}$, then $T|B_e$ is faithful. Let $S$ be a faithful $\ast$-representation of $A$, then $T \otimes S$ is faithful representation of the unit fiber of $\mathscr{B} \otimes_{min}A$, i.e $B_e \otimes_{min}A$, thus 
\begin{equation*}
{\rm{Ind}}_{B_e \otimes_{min}A \uparrow \mathscr{B} \otimes_{min}A}((T|B_e) \otimes S)={\rm{Ind}}_{B_e \uparrow \mathscr{B}}(T|B_e) \otimes S
\end{equation*}
 is weakly equivalent to regular representation of $\mathscr{B} \otimes_{min}A$. On the other hand, ${\rm{Ind}}_{B_e \uparrow \mathscr{B}}(T|B_e)$ is (weakly equivalent to) regular representation of $\mathscr{B}$ and $S$ is faithful $\ast$-representation of $A$, thus $C_r^{\ast}(\mathscr{B} \otimes_{min}A)$ and $C_r^{\ast}(\mathscr{B}) \otimes_{min}A$ have the same $\ast$-source $\Gamma$, they are $\ast$-isomorphic.
\end{proof}

\begin{corollary}\label{cnjfdiiffjjfdc}
If $C^{\ast}(\mathscr{B})$ is nuclear, then $B_e$ is nuclear.
\end{corollary}
\begin{proof}
Let $A$ be a $C^{\ast}$-algebra, and $\sigma$ any $C^{\ast}$-norm of $B \otimes A$. Let $S$ and $R$ be non-degenerate faithful $\ast$-representation of $\mathscr{B}$ and $A$ respectively, and $T$ be faithful $\ast$-representation of $B \otimes_{\sigma} A$. Then ${\rm{Ind}}_{B_e \otimes_{\sigma} A \uparrow \mathscr{B} \otimes_{t(\sigma)}}$, which is $\ast$-representation of $C^{\ast}(\mathscr{B}) \otimes_{max} A$ by Proposition \ref{cnhjousd}, is weakly contained in $S \otimes R$ because $C^{\ast}(\mathscr{B})$ is nuclear. Thus by \cite[XI.11.3]{MR936628} we have
\begin{equation*}\begin{split}
T &\leq {\rm{Ind}}_{B_e \otimes_{\sigma} A}(T)
\\& \leq (S \otimes R)| (B_e \otimes_{\sigma} A)
\\&=S|B_e \otimes R,
\end{split}\end{equation*} 
so $\sigma|(B_e \otimes A)$ is equivalent to the minimal $C^{\ast}$-norm of $B_e \otimes A$, this proved that $B_e$ is nuclear.
\end{proof}

Combine Corollary \ref{cnjfdiiffjjfdc} and Proposition \ref{cjjfifshlfgslifs}, we have:
\begin{proposition}\label{utrjkrlsfd}
Let $\mathscr{B}$ be an amenable $C^{\ast}$-algebraic bundle over $G$. Then the $C^{\ast}$-algebra $C^{\ast}(\mathscr{B})=C_r^{\ast}(\mathscr{B})$ is nuclear if and only if: (i) $B_e$ is nuclear; (ii) for any $C^{\ast}$-algebra $A$, $\mathscr{B} \otimes_{min}A$ is amenable. 
\end{proposition}

\section{Approximation Property of $C^{\ast}$-Algebraic Bundles}

Recall from Exel and Ng \cite{MR1891686}, for any $\alpha, \beta \in \mathscr{L}_2(G, B_e)$ we can define a map $\Phi_{\alpha, \beta}: B \to B$ by
\begin{equation*}
\Phi_{\alpha, \beta}(b)=\alpha \cdot b \cdot \beta=\int_G \alpha(x)^{\ast}b \beta(\pi(b)^{-1}x)dx \in B_{\pi(b)} \ \ \ \ (b \in B)
\end{equation*}
and by \cite[Lemma 3.2]{MR1891686} $\Phi$ induces a map $\Psi_{\alpha, \beta}: \mathscr{L}(\mathscr{B}) \to \mathscr{L}(\mathscr{B})$ defined by
\begin{equation*}
\Psi_{\alpha, \beta}(f)(y)=\int_G \alpha(x)^{\ast}f(y)\beta(y^{-1}x)dx=\Phi_{\alpha, \beta}(f(y)) \ \ \ \ (f \in \mathscr{L}(\mathscr{B})).
\end{equation*}

\begin{definition}\label{vnilfsdcsasasd}
$\mathscr{B}$ is said to have AP (i.e approximation property) if there is $M>0$ and nets $\{\alpha_i\}, \{\beta_i\} \subset \mathscr{L}(G, B_e)$ such that: (i) ${\rm{sup}}_i \|\alpha_i\|\|\beta_i\| \leq M$; (ii) $\Psi_{\alpha_i, \beta_i}(b) \to b$ uniformly  on compact slices of $B$ (\cite[Definition 3.6]{MR1891686}).   
\end{definition}

Let $T$ be any non-degenerate $\ast$-representation of $\mathscr{B}$ such that $r=T|B_e$ is faithful, let $\mu_{T}=T \otimes R^G$. For each $\gamma \in \mathscr{L}(G, B_e)$ we define $V^T_{\gamma}: X(T) \to \mathscr{L}_2(G,X(T))$ by
\begin{equation*}
V^T_{\gamma}(\xi)=r(\gamma(x))(\xi) \ \ \ \ (\xi \in X(T)).
\end{equation*}
By the proof of \cite[Lemma 3.1]{MR1891686}, for any $\alpha, \beta \in  \mathscr{L}(G, B_e)$ we have
\begin{equation}\label{ruhunhjvxf}
\widetilde{T}(\Psi_{\alpha,\beta}(f))(\xi)=(V^{T}_{\alpha})^{\ast}\widetilde{\mu_{T}}(f)V^T_{\beta}(\xi) \ \ \ \ (\xi \in X(T)).
\end{equation}
This motivated us to give the following definition:

\begin{definition}\label{tjfdlksjlkfd}
We say that $\mathscr{B}$ have Ultra-Approximation Property (UAP) if there is a net $\{\Psi_i\}$ of maps $\Psi_i: \mathscr{L}_1(\mathscr{B}) \to  \mathscr{L}_1(\mathscr{B})$ such that that for any non-degenerate $\ast$-representation $T$ of $\mathscr{B}$ there are nets $\{V_i\}_{i \in I}, \{W_i\}_{i \in I} \subset \mathcal{O}(X(T), \mathscr{L}_2(G))$ satisfy the following:

i. we have
\begin{equation}\label{vbnhjdfivnfer}
\widetilde{T}(\Psi_i(f))=W_i^{\ast} \widetilde{\mu_T}(f)V_i \ \ \ \ (f \in \mathscr{L}_1(\mathscr{B})).
\end{equation}
Furthermore, if $R \in \mathcal{O}(X(T))$ is in the commuting algebra of $T$ we have
\begin{equation}\label{vbhodfiufd}
V_iR=(R \otimes 1_{\mathcal{O}(\mathscr{L}_2(G))})V_i;
\end{equation}

ii. there is constant $M>0$ such that $\|V_i\|, \|W_i\|\leq M$ for all $i$;

iii. For any $f \in \mathscr{L}_1(\mathscr{B})$, $\Psi_i(f) \to f$ in the norm of $C^{\ast}(\mathscr{B})$.

If these conditions hold, we say that $\{W_i\}_{i \in I}, \{V_i\}_{i \in I}$ are nets of $\{\Psi_i\}_{i \in I}$ under $T$.
\end{definition}

It is easy to prove that if $\mathscr{B}$ has UAP then $\mathscr{B}$ is amenable, and by (\ref{ruhunhjvxf}) if $\mathscr{B}$ has AP it has UAP.

\begin{remark}\label{runglnkdff}
To construct UAP, we usually define a map in a dense subset of $\mathscr{L}_1(\mathscr{B})$ and then extend it. In order to accomplish this, we need the following general easy observation: Let $\mathscr{B}$ be an arbitrary $C^{\ast}$-algebraic bundle over $G$, and $T$ be a faithful $\ast$-representation of $C^{\ast}(\mathscr{B})$. Then by \cite[VIII.16.4]{MR936629} we can identify each $f \in \mathscr{L}_1(\mathscr{B})$ with $T_f$. Now let $M$ be a norm-dense subset of $\mathscr{L}_1(\mathscr{B})$, $F_1: M \to \mathscr{L}_1(\mathscr{B})$ be a map. If $F_2: T(\mathscr{L}_1(\mathscr{B})) \to \mathcal{O}(X(T))$ is a continuous map  such that $F_2(T_f)=T_{F_1(f)}$ for each $f \in M$, then we can conclude that $F_2(T(\mathscr{L}_1(\mathscr{B}))) \subset T(\mathscr{L}_1(\mathscr{B}))$, thus $F_2$ is a map from $\mathscr{L}_1(\mathscr{B})$ into $\mathscr{L}_1(\mathscr{B})$ as an extension of $F_1$, furthermore it is easy to verify that $F_1$ is continuous with respect to the norm of $\mathscr{L}_1(\mathscr{B})$.
\end{remark}

\begin{proposition}\label{vnnisfdliasiooa}
If $\mathscr{B}$ has UAP, then $\mathscr{D}$ has UAP, in particular it is amenable.
\end{proposition}
\begin{proof}
Let $\{\Phi_i\}_{i \in I}$ be UAP of $\mathscr{B}$. Let $\Gamma$ be the linear span of the cross-sections with the form $x \mapsto f(x) r$ for some $f \in\mathscr{L}(\mathscr{B})$ and $r \in \mathscr{C}_0(M)$. Define $\Phi': \Gamma \to \mathscr{L}_2(\mathscr{D})$ by
\begin{equation*}
\Phi'_i(fr)=\Phi_i(f)r \ \ \ \  f \in\mathscr{L}(\mathscr{B}), r \in \mathscr{C}_0(M).
\end{equation*}

Let $\langle T, R \rangle$ be a faithful $\ast$-representation of $\mathscr{D}$, and $\{V_i\}_{i \in I}$ and $\{W_i\}_{i \in I}$ be the nets of $\{\Phi_i\}$ under $T$. Define
\begin{equation*}
\phi: \langle T, R \rangle(\mathscr{L}_1(\mathscr{D}) \to \mathcal{O}(X(\langle T, R \rangle))
\end{equation*}
by
\begin{equation}\label{djlfdslkvf}
\phi(\langle T, R \rangle(fr))=(W_i)^{\ast}{\mu}_{\langle T, R \rangle}(fr)V_i \ \ \ \ (f \in \mathscr{L}(\mathscr{B}), r \in \mathscr{C}_0(M)).
\end{equation}
By Remark \ref{runglnkdff} if we can prove that 
\begin{equation}\label{eijlmldtmvrg}
\phi(\langle T, R \rangle(\Phi'_i(fr))=T(\Phi_i(f)) R_r \ \ \ \ (f \in \mathscr{L}(\mathscr{B}), r \in \mathscr{C}_0(M))
\end{equation}
then each $\Phi'_i$ is extendable to a map from $\mathscr{L}_1(\mathscr{B})$ into $\mathscr{L}_1(\mathscr{B})$. (\ref{eijlmldtmvrg}) is derived by the following:
\begin{equation*}\begin{split}
\phi(\langle T, R \rangle(fr))&=(W_i)^{\ast}{\mu}_{\langle T, R \rangle}(fr)V_i
\\&=(W_i)^{\ast} \mu_T(f) R_r \otimes 1_{\mathcal{O}(\mathscr{L}_2(G))}V_i
\\&=(W_i)^{\ast} \mu_T(f)V_iR_r
\\&={T}(\Phi_i(f))R_r.
\end{split}\end{equation*}
Now let us verify that $\Phi'_i$ satisfies \ref{tjfdlksjlkfd}(i). Let $\langle T, R \rangle$ be arbitrary $\ast$-representation of $\mathscr{B} \otimes_{r}A$, we have
\begin{equation*}\begin{split}
\langle T, R \rangle(\Phi'_i(fr))&=T(\Phi_i(f))R_r
\\&=(W_i)^{\ast}\mu_T(f)V_iR_r
\\&=(W_i)^{\ast}\mu_T(f)R_ar\otimes 1_{\mathcal{O}(\mathscr{L}_2(G))} V_i
\\&=(W_i)^{\ast} \mu_{\langle T, R \rangle}(fr) V_i
\end{split}\end{equation*}
for all $f \in \mathscr{L}(\mathscr{B})$ and $r \in \mathscr{C}_0(M)$, by the linearity and continuity of $\Phi'_i$, (\ref{vbnhjdfivnfer}) is proved. Furthermore, notice that any operator which is in the commuting algebra of $\langle T, R \rangle$ is in the commuting algebra of $T$, (\ref{vbhodfiufd}) holds.
 The verification of \ref{tjfdlksjlkfd}(ii)and \ref{tjfdlksjlkfd}(iii) are routine, we omit them.
\end{proof}

\begin{proposition}\label{cnhjfdskhjfsd}
If $\mathscr{B}$ has UAP, then $\mathscr{B} \otimes_{r}A$ has UAP, in particular it is amenable.
\end{proposition}
\begin{proof}
Let $\{\Phi_i\}_{i \in I}$ be UAP of $\mathscr{B}$. Let $\Gamma$ be the linear span of the cross-sections with the form $x \mapsto f(x) \otimes a$ for some $f \in\mathscr{L}(\mathscr{B})$ and $a \in A$. Define $\Phi_i': \Gamma \to \mathscr{L}_2(\mathscr{B} \otimes_{r}A)$ by
\begin{equation*}
\Phi'_i(f \otimes a)=\Phi_i(f) \otimes a \ \ \ \  f \in\mathscr{L}(\mathscr{B}), a \in A.
\end{equation*}
By the same argument of Proposition \ref{vnnisfdliasiooa} we can prove that each $\Phi'_i$ can be extended to a map from $\mathscr{L}_1(\mathscr{B} \otimes_{r}A)$ to itself, which we still denote by $\Phi'_i$, such that $\{\Phi'_i\}_{i \in I}$ is the UAP of $\mathscr{B} \otimes_{r}A$.
\end{proof}

\begin{remark}\label{vkjliidasdsjigfo}
In Proposition \ref{cnhjfdskhjfsd}, if $\mathscr{B}$ has AP, it is hard to check whether $B \otimes_{r}A$ has AP because we do not know how to identify compact slices in $\mathscr{B} \otimes_{r}A$ in order to check \ref{vnilfsdcsasasd} (ii). The same difficulty occurs in the study of the transformation bundle, and even worst it is hard to see how to define the AP net of $\mathscr{D}$ according to the AP net of $\mathscr{B}$. In this sense, UAP is a more economic concept than AP.

Combine Proposition \ref{vnnisfdliasiooa} and Proposition \ref{cnhjfdskhjfsd}, we can construct many ``non-trivial'' examples of $C^{\ast}$-algebraic bundles which have UAP. For instance, let $(G,M)$ be an amenable $G$-transformation group, by Claire Anantharaman-Delaroche \cite[ Lemma 2.4]{ejjlklds}  it is easy to verify that the semi-direct product bundle $(\mathscr{C}_0(M), G)$ has AP, thus has UAP. Now for any $C^{\ast}$-algebra $A$ and anothoer transformation group $(G,M')$, we can form tensor product of $(\mathscr{C}_0(M), G)$ and $A$, and furthermore the transformation bundle over $G$ derived from this tensor product, all of them have UAP, so they are all amenable. But it is difficult to check whether they have AP. 
\end{remark}

The following is our main theorem:

\begin{theorem}\label{vnnjsdliclsasod}
Let $\mathscr{B}$ be a saturated $C^{\ast}$-algebraic bundle over $G$ with UAP (in particular if $\mathscr{B}$ has AP). Then for any closed subgroup $H \subset G$ the restriction bundle $\mathscr{B}_H$ is amenable, and $C^{\ast}(\mathscr{B}_H)=C_r^{\ast}(\mathscr{B}_H)$ is nuclear if and only if $B_e$ is nuclear.
\end{theorem}
\begin{proof}
By Theorem \ref{vnkjlslinvrfd} and Proposition \ref{vnnisfdliasiooa} $\mathscr{B}_H$ is amenable. The proof of the other part is the combination of Theorem \ref{vnkjlslinvrfd}, Proposition \ref{utrjkrlsfd}, Proposition \ref{vnnisfdliasiooa} and Proposition \ref{cnhjfdskhjfsd}.
\end{proof}

The `if' part of the following corollary is well-known in varied specific forms:

\begin{corollary}
If $G$ is amenable locally compact group and $\mathscr{B}$ is a saturated $C^{\ast}$-algebraic bundle over $G$, then $C^{\ast}(\mathscr{B})=C_r^{\ast}(\mathscr{B})$ is nuclear if and only if $B_e$ is nuclear. 
\end{corollary}

\bibliographystyle{plain}

\bibliography{referencelist}

\begin{bibdiv}
\begin{biblist}

\bib{vjnj}{article}{
      author={Abadie, Fernando},
      author={Buss, Alcides},
      author={Ferraro, Damián},
       title={Morita enveloping {F}ell bundles},
        date={2017},
        note={arXiv:1711.03013v2},
}

\bib{nkjdklkwa}{article}{
      author={Abadie, Fernando},
      author={Buss, Alcides},
      author={Ferraro, Damián},
       title={Amenability and approximation properties for partial actions and
  {F}ell bundles},
        date={2019},
        note={arXiv:1907.03803v1},
}

\bib{hhncjdl}{article}{
      author={Abadie-Vicens, Fernado},
       title={Tensor products of {F}ell budles over discrete groups},
        date={1997},
        note={arXiv:1301.6883v1},
}

\bib{ejjlklds}{article}{
      author={Anantharaman-Delaroche, Claire},
       title={Amenability and exactness for dynamical systems and their
  {$C^{\ast}$}-algebras},
        date={2002},
     journal={Transactions of the American Mathematical Society},
      volume={354},
       pages={4153\ndash 4178},
}

\bib{aaab}{article}{
      author={Ara, P.},
      author={Exel, R.},
      author={Katsura, T.},
       title={Dynamical systems of type (m,n) and their
  {C}{$^{\ast}$}-algebras},
        date={2013},
     journal={Ergodic Theory Dynam. Systems},
      volume={33},
}

\bib{vnkj}{article}{
      author={Buss, Alcides},
      author={Echterhoff, Siegried},
      author={Willett, Rufus},
       title={Amenability and weak containment for actions of locally compact
  groups on {$C^{\ast}$}-alegbras},
        date={2020},
        note={arXiv:2003.03469v1},
}

\bib{rjnkllds}{article}{
      author={Echterhoff, Siegfried},
      author={Quigg, John},
       title={Induced coactions of discrete groups on {$C^{\ast}$}-algebras},
        date={1998},
        note={arXiv:math/9801069v1},
}

\bib{sjjfdjncj}{article}{
      author={Echterhoff, Siegfried},
      author={Raeburn, Iain},
       title={Induced {$C^{\ast}$}-algebras, coactions and equivariance in the
  sysmmetric imprimitivity theorem},
        date={2000},
     journal={Math.Proc.Cam.Phil.Soc},
      volume={128},
       pages={327\ndash 342},
}

\bib{MR3699795}{book}{
      author={Exel, Ruy},
       title={Partial dynamical systems, {F}ell bundles and applications},
      series={Mathematical Surveys and Monographs},
   publisher={American Mathematical Society, Providence, RI},
        date={2017},
      volume={224},
        ISBN={978-1-4704-3785-5},
      review={\MR{3699795}},
}

\bib{MR1891686}{article}{
      author={Exel, Ruy},
      author={Ng, Chi-Keung},
       title={Approximation property of {$C^*$}-algebraic bundles},
        date={2002},
        ISSN={0305-0041},
     journal={Math. Proc. Cambridge Philos. Soc.},
      volume={132},
      number={3},
       pages={509\ndash 522},
         url={https://doi.org/10.1017/S0305004101005837},
      review={\MR{1891686}},
}

\bib{MR936628}{book}{
      author={Fell, J. M.~G.},
      author={Doran, R.~S.},
       title={Representations of {$^*$}-algebras, locally compact groups, and
  {B}anach {$^*$}-algebraic bundles. {V}ol. 1},
      series={Pure and Applied Mathematics},
   publisher={Academic Press, Inc., Boston, MA},
        date={1988},
      volume={125},
        ISBN={0-12-252721-6},
        note={Basic representation theory of groups and algebras},
      review={\MR{936628}},
}

\bib{MR936629}{book}{
      author={Fell, J. M.~G.},
      author={Doran, R.~S.},
       title={Representations of {$^*$}-algebras, locally compact groups, and
  {B}anach {$^*$}-algebraic bundles. {V}ol. 2},
      series={Pure and Applied Mathematics},
   publisher={Academic Press, Inc., Boston, MA},
        date={1988},
      volume={126},
        ISBN={0-12-252722-4},
         url={https://doi.org/10.1016/S0079-8169(09)60018-0},
        note={Banach $^*$-algebraic bundles, induced representations, and the
  generalized Mackey analysis},
      review={\MR{936629}},
}

\bib{hsfh}{article}{
      author={He, Weijiao},
       title={Herz-schur multipliers of {F}ell bundles},
        date={2020},
        note={arXiv:2001.07051v1},
}

\bib{dhhldsw}{article}{
      author={He, Weijiao},
       title={Tensor products of {$C^{\ast}$}-algebraic bundles},
        date={2020},
        note={In preparing},
}

\bib{ijfjkfw}{article}{
      author={Kaliszewski, S.},
      author={Muhly, Paul~S.},
      author={Quigg, John},
      author={Williams, Dana~P.},
       title={Fell bundles and imprimitivity theorems: {M}ansfield's and
  {F}ell's theorems},
        date={2013},
     journal={J.Aust.Math.Soc},
      volume={95},
       pages={68\ndash 75},
}

\bib{cnkjllkew}{article}{
      author={Lalonde, Scott~M.},
       title={Some consequances of the stabilization theorem for {F}ell bundles
  over exact groupoids},
        date={2017},
        note={arXiv:1710.03808v1},
}

\bib{MR3860401}{article}{
      author={McKee, Andrew},
      author={Skalski, Adam},
      author={Todorov, Ivan~G.},
      author={Turowska, Lyudmila},
       title={Positive {H}erz-{S}chur multipliers and approximation properties
  of crossed products},
        date={2018},
        ISSN={0305-0041},
     journal={Math. Proc. Cambridge Philos. Soc.},
      volume={165},
      number={3},
       pages={511\ndash 532},
         url={https://doi.org/10.1017/S0305004117000639},
      review={\MR{3860401}},
}

\bib{fhkldsll}{article}{
      author={Sims, Aidan},
      author={Williams, Dana~P.},
       title={Amenability for {F}ell bundles over groupoids},
        date={2012},
        note={arXiv:1201.0792v1},
}

\bib{cnnjlilid}{article}{
      author={Sims, Aidan},
      author={Williams, Dana~P.},
       title={An equivalence theorem for reduced {F}ell bundle
  {$C^{\ast}$}-alegbras},
        date={2013},
        note={arXiv:1111.5753v1},
}

\bib{aaaaa}{article}{
      author={Takeishi, Takuya},
       title={On nuclearity of {$C^{\ast}$}-algebras of {Fell} bundles over
  $\acute{E}$tale groupoids},
        date={2013},
        note={arXiv:1301.6883v1},
}

\end{biblist}
\end{bibdiv}

\end{document}